\documentclass[leqno]{amsart}
\usepackage{amsmath}
\usepackage{amssymb}
\usepackage{amsthm}
\usepackage{enumerate}
\usepackage[mathscr]{eucal}
\theoremstyle{plain}
\usepackage{tikz}
\newtheorem{theorem}{Theorem}[section]

\theoremstyle{definition}
\newtheorem{definition}[theorem]{Definition}
\newtheorem{remark}[theorem]{Remark}

\newtheorem{example}[theorem]{Example}

\theoremstyle{remark}

\begin{document}

\title{On the numerical Index of polyhedral Banach Spaces}

\author{Debmalya Sain, Kallol Paul, Pintu Bhunia and Santanu Bag }

\address{(Sain) Department of Mathematics, Indian Institute of Science, Bengaluru 560012,
Karnataka, India}
\email{saindebmalya@gmail.com}

\address{(Paul)Department of Mathematics, Jadavpur University, Kolkata 700032, India}
\email{kalloldada@gmail.com}

\address{(Bhunia)Department of Mathematics, Jadavpur University, Kolkata 700032, India}
\email{pintubhunia5206@gmail.com}

\address{(Bag)Department of Mathematics, Vivekananda College For Women,Barisha , Kolkata 700008, India}
\email{santanumath84@gmail.com}

\thanks{The research of Dr. Debmalya Sain is sponsored by Dr. D. S. Kothari Post-doctoral Fellowship
under the mentorship of Professor Gadadhar Misra. Dr. Sain feels elated to acknowledge the
delightful company and great hospitality of Mr. Anirban Dey, a former Bengal Champion in Table
Tennis. Mr. Pintu Bhunia would like to thank UGC, Govt. of India for the financial support in the form of fellowship. The 
authors are grateful to Vladimir Kadets for his helpful suggestions.}




\date{}
\maketitle
\begin{abstract}
 The computation of the numerical index of a Banach space is an intriguing problem, even in case of two-dimensional real polyhedral Banach spaces. In this article we present a general method to estimate the numerical index of any finite-dimensional real polyhedral Banach space, by considering the action of only finitely many  functionals, on the unit sphere of the space.  We further obtain the exact numerical index of a family of $3$-dimensional polyhedral Banach spaces for the first time, in order to illustrate the applicability of our method.
\end{abstract}

\section{Introduction}
The purpose of the present article is to present a general method to study the numerical index of polyhedral Banach spaces. As a concrete application of our approach, we explicitly compute the numerical index of a family of $3$-dimensional polyhedral Banach spaces whose unit balls are obtained by gluing two pyramids at the opposite base faces of a $3$-dimensional right prism, that remained unexplored previously. Let us first establish the notations and the terminologies to be used in the present article.

\smallskip

Let $\mathbb{X}$ be an $n$-dimensional  Banach space. Throughout the article, without explicitly mentioning hereafter, we work with only real Banach spaces. Let $ B_{\mathbb{X}}= \{ x \in \mathbb{X}~:\|x\| \leq 1 \} $ and $  S_{\mathbb{X}}= \{ x \in \mathbb{X}~:\|x\| = 1 \} $ be the unit ball and the unit sphere of $ \mathbb{X} $ respectively. Given a convex set $A$ in $ \mathbb{X}, $  let $ext~A$ denote the set of extreme points of $ A. $ Let $L(\mathbb{X})$ denote the collection of all bounded linear operator from $\mathbb{X}$ to $ \mathbb{X} $ and let $\mathbb X^*$ denote the dual space of $\mathbb{X}.$ Given a bounded linear operator $ T \in L(\mathbb{X}), $ the numerical range $ V(T) $ of $ T $ is defined \cite{BFL, LG} as $ V(T)= \{ x^*(Tx) : x\in S_\mathbb X,x^* \in S_{\mathbb X^*},x^*(x)=1\}$. The numerical radius $ v(T) $ of $ T $ is defined as $ v(T)=\sup \{ | x^*(Tx) |~:x\in S_\mathbb X,x^* \in S_{\mathbb X^*},x^*(x)=1\}.$ It is easy to see that $ v(T)=\sup \{ | x^*(Tx) |~:x \in \textit{ext}~ B_\mathbb X, x^* \in \textit{ext}~B_{\mathbb X^*}, x^*(x)=1\}. $ Finally, the numerical index $ n(\mathbb{X}) $ of $ \mathbb{X} $ is defined as $ n(\mathbb{X})=\inf \{ v(T):T\in L(\mathbb{X}),~ \| T \| =1\}. $ The numerical index is an important geometric constant associated with a given Banach space. Its properties have been studied by several authors \cite{MJ, MJA, MR, MV} in the context of various Banach spaces. In \cite{MJ} the authors have explicitly computed the numerical index of certain polyhedral norms in the $2$-dimensional case. $ \mathbb{X} $ is said to be a polyhedral Banach space if $ B_{\mathbb{X}} $ has only finitely many extreme points. Equivalently, $ \mathbb{X} $ is a polyhedral Banach space if $ B_{\mathbb{X}} $ is a polyhedron. In this context, let us mention the following definitions which are relevant to our present work:
 
\begin{definition}
A polyhedron $P$ is a non-empty compact subset of $\mathbb{X}$ which is the intersection of finitely many closed half-spaces of $\mathbb{X}$, i.e., $P=\bigcap_{i=1}^r M_i,$ where $ M_i $ are closed half-spaces in $ \mathbb{X} $ and $r \in \mathbb{N}.$ The dimension $  \dim P $ of the polyhedron $P$ is defined as the dimension of the subspace generated by the differences $ v-w $ of vectors $ v,w \in P. $
\end{definition}

\begin{definition}
A polyhedron $ Q $ is said to be a face of the polyhedron $ P $ if either $ Q=P $ or if we can write $ Q = P \cap \delta M, $ where $ M $ is a closed half-space in $ \mathbb{X} $ containing $ P $ and $ \delta M $ denotes the boundary of $  M. $ If $ \dim Q = i, $ then $  Q $ is called an $i$-face of $ P. $ $ (n-1)$-faces of $ P $ are called facets of $ P $ and $1$-faces of $ P $ are called edges of $ P. $
\end{definition}

As mentioned in \cite{VMR}, explicit computation of the numerical index of Banach spaces is a difficult problem in general, that remains unsolved for many classical spaces. In this article we propose a simple approach towards studying the numerical index of polyhedral Banach spaces. We show that for polyhedral Banach spaces, it is possible to consider only finitely many distinguished functionals acting on the space, in order to have a reasonably good estimate of the numerical index of the concerned space. To serve our purpose, we introduce the following definition in the study of the numerical index of polyhedral Banach spaces:

\begin{definition}
Let $ \mathbb{X} $ be a finite-dimensional polyhedral Banach space. Let $ F $ be a facet of the unit ball $ B_{\mathbb{X}} $ of $  \mathbb{X}. $ A functional $ f \in S_{\mathbb{X}^{*}} $ is said to be a \emph{supporting functional corresponding to the facet $ F $ of the unit ball $ B_{\mathbb{X}} $} if the following two conditions are satisfied:\\
$ (1) $ $ f $ attains norm at some point $ v $ of $ F, $\\
$ (2) $ $ F =(v + \ker f)\cap S_{\mathbb{X}}. $
\end{definition}

It is easy to see that there is a unique hyperspace $H$ such that an affine hyperplane to $ H $  contains the facet $F$ of the unit ball $ B_{\mathbb{X}}.$ Moreover, there exists a unique norm one linear functional $ f, $ such that $ f $ attains norm on $ F $ and $ker f = H.$ In particular, $f$ is a supporting functional to $ B_{\mathbb{X}} $ at every point of $ F. $ It is immediate that there are only finitely many  facets for the unit ball of a polyhedral Banach space and accordingly there are only finitely many supporting functionals corresponding to the facets of the unit ball. We study the actions of these functionals on the unit sphere of the polyhedral Banach space in order to estimate the numerical index of the space. We prove that the supporting functionals corresponding to the facets of the unit ball of a polyhedral Banach space are the only extreme points of the unit ball of the dual space. In this connection, we require the notions of the polar of a set and the prepolar of a set in a Banach space. Let $ A \subset \mathbb{X} $ and let $ B \subset \mathbb{X}^{*}. $ Then the polar of $ A $ is defined as $A^{o}=\{f\in \mathbb{X}^{*}: |f(x)| \leq 1 ~~\forall x \in A\}$ and  the prepolar of $B$ is defined as $^oB=\{ x\in \mathbb{X} :~ | f(x) | \leq 1~~  \forall f\in B\}$. We refer the readers to the excellent recent book \cite{K}, for more information on this. Furthermore, we prove that in many important cases, it is possible to explicitly compute the exact numerical index of the concerned Banach space by this method. Indeed, we determine the numerical index of the family of $3$-dimensional polyhedral Banach spaces whose unit balls are oblique prisms with regular polyhedrons as its base. As another application of this novel method,  we compute the numerical index of the family of $3$-dimensional polyhedral Banach spaces whose unit balls are obtained by gluing two pyramids at the opposite base faces of a $3$-dimensional right prism.

\section{Main Results}

We begin this section by proving that the supporting functionals corresponding to the facets of the unit ball of a finite-dimensional polyhedral Banach space are the only extreme points in the unit ball of the dual space.

\begin{theorem}\label{theorem:extreme}
Let $\mathbb{X}$ be an $n$-dimensional polyhedral Banach space. Then $f \in S_{\mathbb{X}^{*}}$ is an extreme point of $ B_{\mathbb{X}^{*}} $ if and only if $ f $ is a supporting functional corresponding to a facet of $B_{\mathbb{X}}.$  
\end{theorem}

\begin{proof}
First we prove the sufficient part of the theorem.  If possible, suppose that $ f $ is a supporting functional corresponding to a facet $F$ of $B_{\mathbb{X}}$  such that $f$ is not an extreme point of $ B_{\mathbb{X}^{*}}. $ Then $f$ can be written as $f=(1-t)f_1+tf_2$, for some $f_1, f_2 \in S_{\mathbb{X^*}} \setminus \{f\}$ and for some $t\in (0,1)$. Now, $f$ attains norm at some extreme point $v $ of $B_{\mathbb{X}}$ and so $ (v+\ker f) $ is a supporting hyperplane to $ B_{\mathbb{X}} $ at the point $ v. $ Now, $ F(\subsetneq (v+\ker f)) $ being a facet, contains $ n $ linearly independent vectors $ w_1, w_2, \ldots, w_n \in S_{\mathbb{X}} $. It is easy to see that for each $ i=1,2,\ldots,n, $ we have, $ f(w_i)=1. $ Therefore, we have, $ 1 = f(w_i) = (1-t)f_1(w_i)+tf_2(w_i). $ As $ \| f_1 \|=\| f_2 \|=\|w_i\| =1, $ it follows that $ f_1(w_i)= f_2(w_i)=f(w_i)=1.$ Since $ w_1,w_2,\ldots,w_n $ are linearly independent, we must have $ f_1=f_2=f. $ This contradiction completes the proof of the sufficient part of the theorem.

Now we prove the necessary part of the theorem. Consider the set $B=\{ f\in S_{\mathbb{X^*}} : ~~f $ is a supporting functional corresponding to a facet of $ B_{\mathbb{X}} \}$. Then, it is easy to see that $^oB=B_{\mathbb{X}}$ and $(^oB)^o=B_{X^*}.$ Now, by  the Bipolar theorem, we have that $(^oB)^o$ is the closed convex balanced hull of $B$. Therefore, every supporting functional can be represented as a convex combination of points of $B$. This shows that the extreme points of $ B_{\mathbb{X}^{*}} $ are contained in $B$.  This completes the proof of the necessary part of the theorem and establishes it completely.

\end{proof}

In view of the above Theorem \ref{theorem:extreme}, it is evident that supporting functional corresponding to the facets of the unit ball of a polyhedral Banach space play a very special role in determining the numerical index of the space. In the following theorem, we apply this idea to obtain a reasonably good estimate of the numerical index of any finite-dimensional polyhedral Banach space. 

\begin{theorem}\label{theorem:estimate}

Let $\mathbb{X}$ be an $n$-dimensional polyhedral Banach space. Let $ \pm v_1, \pm v_2,$ $\ldots, \pm v_m $ be the vertices of $ B_{\mathbb{X}}. $ For each $ i=1,2,\ldots,m, $ let $ F_{i1},F_{i2},\ldots,F_{in} $ be any $ n $ facets of $ B_{\mathbb{X}} $ meeting at $ v_i. $ Let $ f_{i1},f_{i2}, \ldots, f_{in} $ be  the $ n $ supporting functionals corresponding to the facets $ F_{i1},F_{i2},\ldots,F_{in} $ respectively. Let \[ \kappa_i = \min_{x \in S_{\mathbb{X}}} \max_{1 \leq r \leq n} {\| f_{ir}(x)\| }. \] 
Then each $\kappa_i > 0$ and the numerical index $ n(\mathbb{X}) $ of $ \mathbb{X} $ can be estimated as follows:
\[n(\mathbb{X}) \geq \min \{\kappa_1,\kappa_2,\ldots,\kappa_m\}.\] 
\end{theorem}

\begin{proof}
It is easy to observe that for each $i=1,2,\ldots,m,$ we have, $ {\bigcap}_{r=1}^n \ker f_{ir}  = \{0\}. $ Therefore, given any $ x \in S_{\mathbb{X},} $ we have, $ \max_{1 \leq r \leq n} \{ \vert f_{ir}(x)\vert \} > 0. $ Since $ x \longrightarrow \max_{1 \leq r \leq n}\{\vert f_{ir}(x)\vert \} $ is a continuous function from $S_{\mathbb{X}}$ to $\mathbb{R}$ and $ S_{\mathbb{X}} $ is a compact set, it follows that  each $\kappa_i > 0.$ \\
Let $ T \in S_{L(\mathbb{X})} $ be arbitrary. Since $ T $ must attain norm at an extreme point of $ B_{\mathbb{X}}, $ there exists $ v_{j} $ such that $ \|Tv_{j}\|=\| T \|=1 $ for some $j \in \{1,2,\ldots,m\}.$ Therefore, we have,
 $ v(T) = \sup\{ |x^*(Tx)|~:x^* \in \textit{ext}~ B_{\mathbb{X}^*},x \in \textit{ext}~ B_{\mathbb{X}},x^*(x)=1\} \geq \max_{1 \leq r \leq n} \{\vert f_{jr}(Tv_{j})\vert \} \geq \kappa_{j}. $
This completes the proof of the theorem.
\end{proof}

In the following example, we illustrate the applicability of Theorem \ref{theorem:estimate} towards estimating the numerical index of a $2$-dimensional polyhedral Banach space.

\begin{example}\label{theorem:2}
Let $\mathbb{X} = \mathbb{R}^{2}$ be a $2$-dimensional polyhedral Banach space such that $B_{\mathbb{X}}$ is an irregular hexagon with vertices $ \pm(1,1), \pm(\frac{1}{2},2), \pm(-1,1). $ Then $n(\mathbb{X})\geq \frac{5}{17}$.
\end{example}

\begin{proof}
Let us denote the vertices of $B_{\mathbb{X}}$ by $v_i$, $i=1,2,\ldots,6$, where, $v_1=(1,1), v_2=(\frac{1}{2},2), v_3=(-1,1), v_4=(-1,-1), v_5=(-\frac{1}{2},-2), v_6=(1,-1)$. The unit sphere $S_{\mathbb{X}}$ is shown in figure 1.
\begin{figure}[ht]
\centering 
\includegraphics[width=0.4\linewidth]{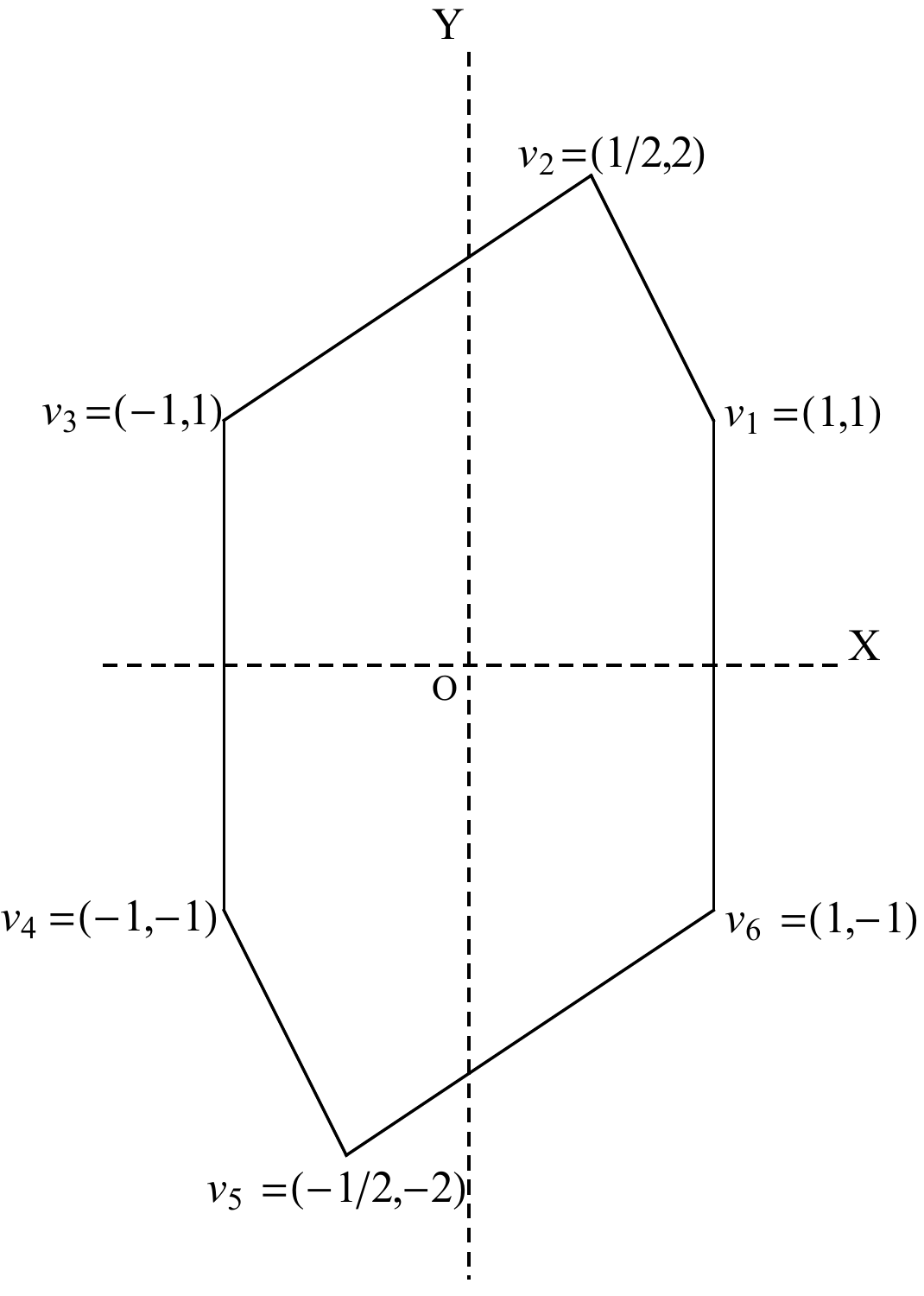}
\caption{}
\label{figure}
\end{figure}
Let $T \in S_{L(\mathbb{X})}$ be arbitrary. Clearly, $T$ attains norm at some $ v_i. $ For each $ i=1,2,\ldots,6, $ let $ f_{i1},f_{i2} $ be the two supporting functionals corresponding to the two edges of $ B_{\mathbb{X}} $ meeting at $ v_i. $ 
It is easy to check that $f_{11},f_{12}$ are given as : $f_{11}(\alpha,\beta)=\frac{2}{3}\alpha+\frac{1}{3}\beta$ and $f_{12}(\alpha,\beta)=\alpha$ $~\forall~(\alpha,\beta) \in \mathbb{X}.$ 
On direct computation, we obtain that $\max\{ \vert f_{11}(\alpha,\beta)\vert,\vert f_{12}(\alpha,\beta)\vert\}\geq \frac{5}{17} = \kappa_1~\forall~(\alpha,\beta)\in S_{\mathbb{X}}$.  
Similarly, it can be shown that $ \kappa_2 = \frac{4}{7} $ and $ \kappa_3 = \frac{9}{13} $. Therefore, using Theorem \ref{theorem:estimate}, we conclude that $ n(\mathbb{X}) \geq  \frac{5}{17}. $
\end{proof}

We would like to note that the basic idea behind estimating the numerical index of a finite-dimensional polyhedral Banach space by applying Theorem \ref{theorem:estimate}, is to consider only finitely many supporting functionals corresponding to the facets of the unit ball. However, to our pleasant surprise, this idea yields the exact value of the numerical index of some particular family of Banach spaces. Mart\'in and Mer\'i \cite{MJ} have computed the numerical index of some $2$-dimensional polyhedral Banach spaces. One of their main results is to compute the numerical index of those $2$-dimensional polyhedral Banach spaces for which the unit balls are regular polyhedrons. In this article, we generalize this result by applying the newly introduced concept of supporting functionals corresponding to the facets of the unit ball. Indeed, we compute the numerical index of the family of $3$-dimensional polyhedral Banach spaces whose unit balls are oblique prisms with regular polyhedrons as the bases. 

\begin{theorem}\label{theorem:odd radius}
Let $\mathbb{X}$  be a $3$-dimensional polyhedral Banach space such that $B_{\mathbb{X}}$ is a prism with vertices $(\cos(j-1)\frac{\pi}{n} \pm{l}, \sin(j-1)\frac{\pi}{n},\pm{1})$,  $j \in \{1,2,\ldots,2n\}$,  where $l \in \mathbb{R}$ and $n$ is odd.  Then $n(\mathbb{X}) = \sin\frac{\pi}{2n}$.

\end{theorem}

\begin{proof} We complete the proof in two steps. In  step 1, we show that if  $T \in S_{L(\mathbb{X})}$ then $v(T) \geq \sin\frac{\pi}{2n}$ and in step 2, we exhibit an operator  $T\in S_{L({\mathbb{X}})}$ such that $v(T) = \sin\frac{\pi}{2n}$.
 
\medskip

\textbf{Step 1.}  Let the vertices of $B_{\mathbb{X}}$ be  $v_{\pm{j}}, j \in \{1,2,\ldots,2n\}$, where $v_{\pm{j}}=(\cos(j-1)\frac{\pi}{n}\pm{l},\sin(j-1)\frac{\pi}{n},\pm{1})$. Let $G_j$ denote the facet of $B_{\mathbb{X}}$ containing $v_j, v_{-j}, v_{j+1}$, where $v_{2n+1}=v_1$. Let $F_1$ denote the facet of $B_{\mathbb{X}}$ containing $v_1,v_2,v_3$ and $F_{-1}$ denote the facet of $B_{\mathbb{X}}$ containing $v_{-1},v_{-2},v_{-3}$. 

Let $T\in S_{L({\mathbb{X}})}.$  We show that there exists $x \in S_{\mathbb{X}}$ such that $|x^*(Tx)|\geq \sin\frac{\pi}{2n}$ where $x^*(x)=1$ and $x^* \in S_{\mathbb{X^*}}$. Clearly $T$ attains norm at some extreme point $x=v_j$ of $B_{\mathbb{X}}$. Assume that $ x = v_1 $, the proof for other vertices will follow similarly. Then $Tx \in S_{\mathbb{X}}$ and so $Tx \in G_j$ for some $j\in \{ 1,2,\ldots,2n\}$ or $Tx \in F_1 \cup F_{-1}$. We consider the following four cases, depending on which part of the unit sphere $Tx$ belongs to.

\begin{figure}[ht]
\centering 
\includegraphics[width=0.5\linewidth]{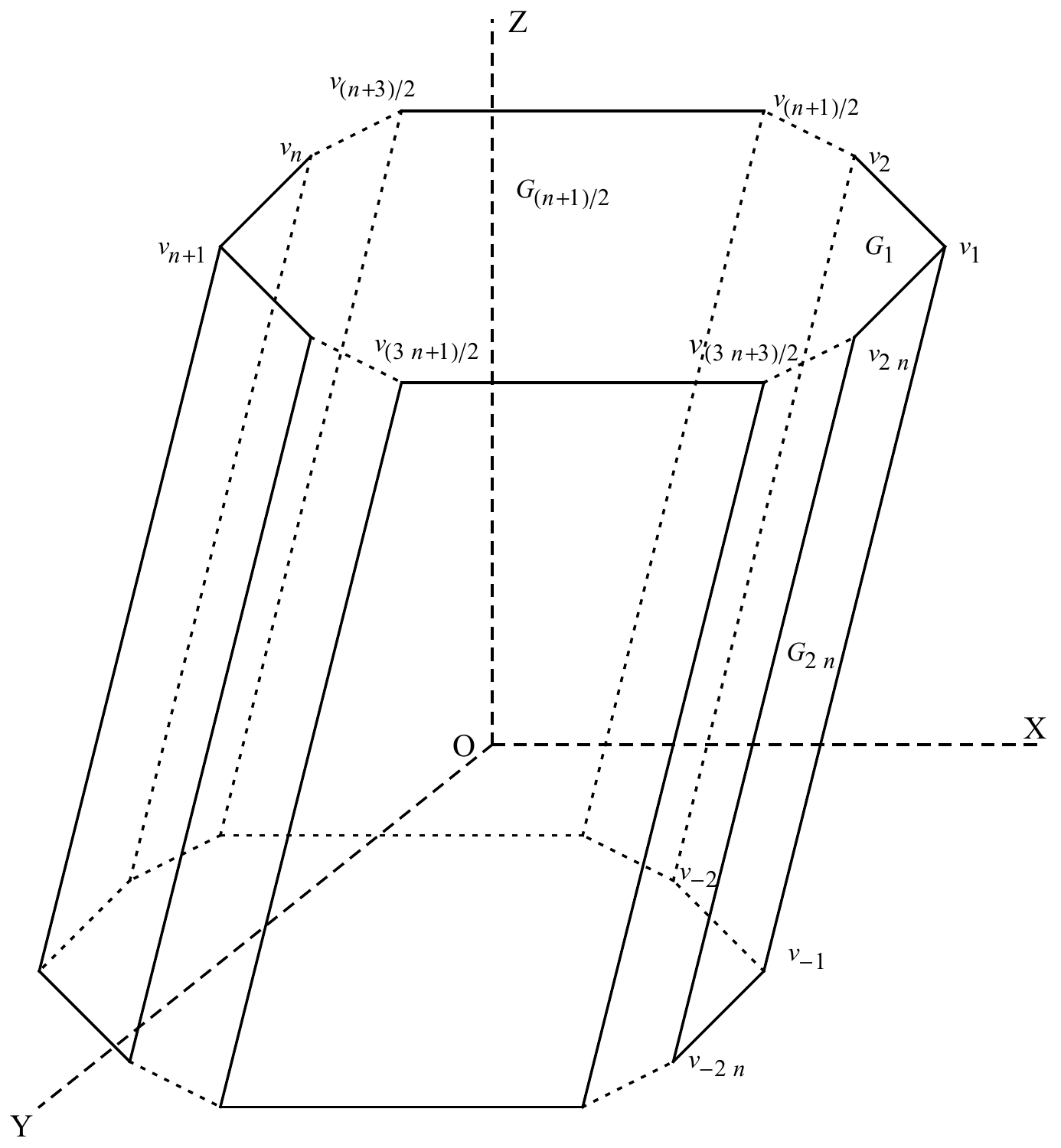}
\caption{}
\label{figure}
\end{figure}

\noindent \textbf{Case 1.} $(i)$ $Tx \in G_j$ for some $ j = 1,2, \ldots, \frac{n-1}{2}.$  Let $Tx=(\alpha,\beta,\gamma),$ then it is easy to see that $\beta=\frac{\cos\frac{\pi}{2n}-(\alpha-l \gamma) \cos(2j-1)\frac{\pi}{2n}}{\sin(2j-1)\frac{\pi}{2n}}$. Let $f_{11}$ be the supporting functional at $v_1$  such that $ (v_1 + ker~f_{11}) \cap S_{\mathbb{X}} = G_1.$  Clearly  $ f_{11}(1+l,0,1) = 1 $ and $ f_{11}(-\sin\frac{\pi}{2n}+l,\cos\frac{\pi}{2n},1) = f_{11}(\sin\frac{\pi}{2n}+l,-\cos\frac{\pi}{2n},1) = 0$.  Let $(\alpha,\beta,\gamma)=\alpha_1(1+l,0,1)+\beta_1(-\sin\frac{\pi}{2n}+l,\cos\frac{\pi}{2n},1)+\gamma_1(\sin\frac{\pi}{2n}+l,-\cos\frac{\pi}{2n},1)$, for some  $\alpha_1, \beta_1, \gamma_1 \in \mathbb{R}$. Then $f_{11}(Tx)=\alpha_1=(\alpha-l\gamma)(1-\frac{\tan\frac{\pi}{2n}}{\tan(2j-1)\frac{\pi}{2n}})+\frac{\sin\frac{\pi}{2n}}{\sin(2j-1)\frac{\pi}{2n}}$. Consider the linear functional $\phi$ on $\mathbb{X}$ as :  $ \phi(\alpha,\beta,\gamma)=\alpha-l\gamma.$ Then $ker~\phi$ is a hyperspace passing through the points  $ \frac {v_{\frac{n+1}{2}}+ v_{\frac{n+3}{2}}}{2}$ , $ \frac {v_{-\frac{n+1}{2}}+ v_{-\frac{n+3}{2}}}{2}$ and $ \frac {v_{\frac{3n+1}{2}}+ v_{\frac{3n+3}{2}}}{2}$ and $ ker~\phi \cap G_j = \emptyset $ for each $j \in \{1,2,\ldots, \frac{n-1}{2}\} .$ Since $ v_1 \in G_j$ and $\phi(v_1)=1 >0$, therefore   for  $(\alpha, \beta, \gamma) \in G_j$, $ \phi( \alpha, \beta, \gamma) = \alpha - l \gamma > 0$ for each $j \in \{1,2,\ldots, \frac{n-1}{2}\} .$ Also for each $j \in \{1,2,\ldots, \frac{n-1}{2}\} $ it is easy to see that $1-\frac{\tan\frac{\pi}{2n}}{\tan(2j-1)\frac{\pi}{2n}} \geq 0 $ and $ 0 < \sin(2j-1)\frac{\pi}{2n} < 1 $ and so $f_{11}(\alpha, \beta, \gamma) = (\alpha-l\gamma)(1-\frac{\tan\frac{\pi}{2n}}{\tan(2j-1)\frac{\pi}{2n}})+\frac{\sin\frac{\pi}{2n}}{\sin(2j-1)\frac{\pi}{2n}} > \sin{\frac{\pi}{2n}}.$\\
\noindent \textbf{Case 1.} $(ii)$  $Tx \in G_j $ for some $ j = n+1,n+2, \ldots, \frac{3n-1}{2}$. 
Clearly $G_{n+j} = -G_j $ for $ j = 1,2,\dots,\frac{n-1}{2}$ and so if $Tx \in G_j $ for some $ j = n+1,n+2, \ldots, \frac{3n-1}{2},$ then  $|f_{11}(Tx)|>\sin\frac{\pi}{2n}$.\\
\noindent \textbf{Case 2.} $(i)$  $Tx\in G_j$ for some $ j = n,n-1, \ldots, \frac{n+3}{2}.$ \\
\noindent \textbf{Case 2.} $(ii)$ $Tx \in G_j $ for some $ j = 2n,2n-1, \ldots, \frac{3n+3}{2}$.\\
Proceeding as in Case 1 we can show that $|f_{12}(Tx)|>\sin\frac{\pi}{2n}$, where $f_{12}$ is the supporting functional at $v_1$ such that $(v_1 + ker~f_{12}) \cap S_{\mathbb{X}} = G_{2n}.$\\
\noindent \textbf{Case 3.}   $Tx\in F_1 \cup F_{-1} $. Let $f_{13}$ be the supporting functional at $v_1$  such that $ (v_1 + ker~f_{13}) \cap S_{\mathbb{X}} = F_1.$   Then it is easy to observe that $|f_{13}(Tx)|=1>\sin \frac{\pi}{2n}$. \\ 
\noindent \textbf{Case 4.}  $Tx \in G_{\frac{n+1}{2}} \cup G_{\frac{3n+1}{2}}$. If $Tx \in G_{\frac{n+1}{2}}$ then $f_{11}(Tx)=\alpha-l\gamma +\sin \frac{\pi}{2n}$ and $f_{12}(Tx)=\alpha-l\gamma-\sin \frac{\pi}{2n}$.  If $\alpha-l\gamma \geq 0$ then $f_{11}(Tx)\geq \sin \frac{\pi}{2n}$. If  $\alpha-l\gamma < 0$ then $f_{12}(Tx) < -\sin \frac{\pi}{2n}$, i.e., $|f_{12}(Tx)|>\sin \frac{\pi}{2n}$. \\
If $Tx \in G_{\frac{3n+1}{2}}$ then $Tx \in -G_{\frac{n+1}{2}}$. Therefore as before $ \mid f_{11}(Tx) \mid \geq \sin \frac{\pi}{2n}$ and $ \mid f_{12}(Tx) \mid \geq \sin \frac{\pi}{2n}$.\\
\noindent Thus in all cases there exists a supporting functional $x^*$ at $x$ so that $ \mid x^*(Tx) \mid \geq \sin \frac{\pi}{2n}$ and so $v(T)\geq \sin \frac{\pi}{2n}$. This completes Step 1.

\medskip

\textbf{Step 2.} Let us consider a linear operator $ T :\mathbb{X}\rightarrow \mathbb{X}$ as
\begin{eqnarray*}
Tv_1=&T(1+l,0,1)=&(l\sin\frac{\pi}{2n},\cos\frac{\pi}{2n},\sin\frac{\pi}{2n}),\\
Tv_2=&T(\cos\frac{\pi}{n}+l,\sin\frac{\pi}{n},1)=&(-\sin\frac{\pi}{n}\cos\frac{\pi}{2n}+l\sin\frac{\pi}{2n},\cos\frac{\pi}{n}\cos\frac{\pi}{2n},\sin\frac{\pi}{2n}), \\
Tv_3=&T(\cos\frac{2\pi}{n}+l,\sin\frac{2\pi}{n},1)=&(-\sin\frac{2\pi}{n}\cos\frac{\pi}{2n}+l\sin\frac{\pi}{2n},\cos\frac{2\pi}{n}\cos\frac{\pi}{2n},\sin\frac{\pi}{2n}). 
\end{eqnarray*}
It is easy to check that  $Tv_j=T(\cos(j-1)\frac{\pi}{n}+l,\sin(j-1)\frac{\pi}{n},1)=(-\sin(j-1)\frac{\pi}{n}\cos\frac{\pi}{2n}+l\sin\frac{\pi}{2n},\cos(j-1)\frac{\pi}{n}\cos\frac{\pi}{2n},\sin\frac{\pi}{2n})$, $\forall j=4,5,\ldots,2n$ and  $Tv_j \in G_{\frac{n+2j-1}{2}}$ $  \forall j=1,2,\ldots,2n$, where $v_{2n+k}=v_k$ and $ G_{2n+k}=G_k$, for $k=1,2,\ldots,2n$. Then $\| T \|=1$. It follows from Theorem \ref{theorem:extreme}  that $ \textit{ext}~B_{\mathbb{X}^*} = \{f_{jk}~:1 \leq j \leq 2n~~ \textit{and}~~ 1\leq k \leq 3\}$ and so $v(T)=\max \{ \vert f_{jk}(Tv_j)\vert~:1 \leq j \leq 2n ~~\mbox{and} ~~1\leq k \leq 3\}$. 
Now consider the action of $f_{j1}$ on $Tv_j$. It is easy to check that $(\cos(j+\frac{n-1}{2})\frac{\pi}{n}+l,\sin(j+\frac{n-1}{2})\frac{\pi}{n},1), (-\cos(j+\frac{n-1}{2})\frac{\pi}{n}+l,-\sin(j+\frac{n-1}{2})\frac{\pi}{n},1) \in \ker f_{j1}$. Let $(v_j=\alpha_{j1}(\cos(j-1)\frac{\pi}{n}+l,\sin(j-1)\frac{\pi}{n},1)+\beta_{j1}(\cos(j+\frac{n-1}{2})\frac{\pi}{n}+l,\sin(j+\frac{n-1}{2})\frac{\pi}{n},1)+\gamma_{j1}(-\cos(j+\frac{n-1}{2})\frac{\pi}{n}+l,-\sin(j+\frac{n-1}{2})\frac{\pi}{n},1)$. Then $f_{j1}(Tv_j)=\alpha_{j1}=\sin\frac{\pi}{2n}$. \\
Next we consider the action of $f_{j2}$ on $Tv_j$. It is easy to check that $(\cos(j+\frac{n-3}{2})\frac{\pi}{n}+l,\sin(j+\frac{n-3}{2})\frac{\pi}{n},1), (-\cos(j+\frac{n-3}{2})\frac{\pi}{n}+l,-\sin(j+\frac{n-3}{2})\frac{\pi}{n},1) \in \ker f_{j2}$. Let $Tv_j=\alpha_{j2}(\cos(j-1)\frac{\pi}{n}+l,\sin(j-1)\frac{\pi}{n},1)+\beta_{j2}(\cos(j+\frac{n-3}{2})\frac{\pi}{n}+l,\sin(j+\frac{n-3}{2})\frac{\pi}{n},1)+\gamma_{j2}(-\cos(j+\frac{n-3}{2})\frac{\pi}{n}+l,-\sin(j+\frac{n-3}{2})\frac{\pi}{n},1)$. Then $f_{j2}(Tv_j)=\alpha_{j2}=-\sin\frac{\pi}{2n}$. \\
Finally, we consider the action of $f_{j3}$ on $Tv_j$. It is easy to check that $(1,0,0),$ $ (\cos\frac{\pi}{n},\sin\frac{\pi}{n},0) \in \ker f_{j3}$. 
Let $Tv_j=\alpha_{j3}(\cos(j-1)\frac{\pi}{n}+l,\sin(j-1)\frac{\pi}{n},1)+\beta_{j3}(1,0,0)+\gamma_{j3}(\cos\frac{\pi}{n},\sin\frac{\pi}{n},0)$. Then $f_{j3}(Tv_j)=\alpha_{j3}=\sin\frac{\pi}{2n}$. \\
Therefore, $v(T)=\max_{1\leq j\leq 2n} \{ \vert f_{j1}(Tv_j)\vert, \vert f_{j2}(Tv_j)\vert, \vert f_{j3}(Tv_j)\vert \}=\sin\frac{\pi}{2n}$. This completes the proof of the theorem.
\end{proof}

The case when $ n $ is even can be treated similarly to obtain the following result:
\begin{theorem}\label{theorem:even}
Let $\mathbb{X}$ be a $3$-dimensional polyhedral Banach space such that $B_{\mathbb{X}}$ is a prism with vertices $(\cos(j-1)\frac{\pi}{n} \pm{l}, \sin(j-1)\frac{\pi}{n},\pm{1})$,  $j \in \{1,2,\ldots,2n\}$,  where $l \in \mathbb{R}$ and $n$ is even. Then $n(\mathbb{X}) = \tan\frac{\pi}{2n}. $ 
\end{theorem}
\begin{remark}\label{remark:generalization} Taking $ l=0 $ in Theorem \ref{theorem:odd radius} and Theorem \ref{theorem:even}, we obtain the numerical index of a $3$-dimensional polyhedral Banach space whose unit ball is a right prism with regular $2n$-gon as its base. We also observe that Theorem \ref{theorem:odd radius} and Theorem \ref{theorem:even} remains unchanged even if the height of the prism is changed. In particular, taking the height of the prism to be $ 0, $ this gives us another method, in comparison with Theorem $ 5 $ of \cite{MJ}, to compute the numerical index of $2$-dimensional polyhedral Banach spaces whose unit balls are regular $2n$-gons. 
\end{remark}

\begin{remark}\label{remark:direct sum}
We would like to further note that Theorem \ref{theorem:odd radius} and Theorem \ref{theorem:even} can also be proved using another approach. In Proposition $ 1 $ of \cite{MR}, Mart\'in and Mer\'i has proved in particular that if $ \mathbb{X},\mathbb{Y} $ are two Banach spaces then $ n(\mathbb{X} \oplus_{\infty} \mathbb{Y}) = \min\{ n(\mathbb{X}),n(\mathbb{Y}) \}.$ Therefore, combining Theorem $ 5 $ of \cite{MJ} with proposition $ 1 $ of \cite{MR}, we obtain the numerical index of the $3$-dimensional polyhedral Banach spaces whose unit balls are right prisms with regular polyhedrons as its base. Finally, we observe that the $3$-dimensional polyhedral Banach space whose unit ball is an oblique prism with regular polyhedron as its base is isometric to the $3$-dimensional polyhedral Banach space whose unit ball is a right prism with regular polyhedron as its base. 
\end{remark}

In view of Remark \ref{remark:direct sum}, we next present the example of a polyhedral Banach space that can not be written as the infinity sum of a lower dimensional polyhedral Banach space with a suitably chosen Banach space. Indeed, we compute the numerical index of such a polyhedral Banach space using the notion of supporting functional corresponding to a facet of the unit ball of the concerned space. This further illustrates the applicability of our method towards computing the numerical index of polyhedral Banach spaces.

\begin{theorem}\label{theorem:pyramid}

Let $\mathbb{X}$  be a $3$-dimensional polyhedral Banach space such that $B_{\mathbb{X}}$ is a polyhedron obtained by gluing two pyramids at the opposite base faces of a right prism having square base, with vertices $ \pm(1,1,1), \pm(-1,1,1), \pm(-1,-1,1), \pm(1,-1,1),$ $ \pm(0,0,2)$.  Then $n(\mathbb{X})= \frac{1}{2}$.
\end{theorem}
\begin{proof}  We complete the proof in two steps. In  step 1, we show that if  $T \in S_{L(\mathbb{X})}$ then $v(T) \geq \frac{1}{2}$ and in step 2, we exhibit an operator  $T\in S_{L({\mathbb{X}})}$ such that $v(T) = \frac{1}{2}$.
 
\medskip

\textbf{Step 1.}
Let the vertices of $B_{\mathbb{X}}$ be $v_{\pm j}, j\in \{1,2,3,4\} $ and $w_{\pm 1}$ where $v_{\pm 1}=(1,1,\pm 1)$, $v_{\pm 2}=(-1,1,\pm 1)$, $v_{\pm 3}=(-1,-1,\pm 1)$, $v_{\pm 4}=(1,-1,\pm 1)$ and $w_{\pm 1}=(0,0,\pm 2)$. Let $G_j$ denote the facet of $B_{\mathbb{X}}$  containing $v_j, v_{j+1}, v_{-j}$ where $v_5=v_1$. Let $F_{\pm j}$ denote the facet of $B_{\mathbb{X}}$  containing $v_{\pm j}, v_{\pm(j+1)}, w_{\pm 1}$ where $v_{\pm 5}=v_{\pm 1}$. 
Let $T\in S_{L({\mathbb{X}})}$.  We show that there exists $x \in S_{\mathbb{X}}$ such that $|x^*(Tx)|\geq \frac{1}{2}$ where $x^*(x)=1$ and $x^* \in S_{\mathbb{X^*}}$. Clearly $T$ attains norm at some extreme point $x$ of $B_{\mathbb{X}}$. Assume that $ x =v_1$ or  $ x =w_1$, the proof for other vertices will follow similarly. Then $Tx \in S_{\mathbb{X}}$ and so $Tx \in G_j$ or $Tx \in F_{\pm j}$  for some $j\in \{ 1,2,3,4\}$ . We consider the following four cases, depending on which part of the unit sphere $Tx$ belongs to.

\begin{figure}[ht]
\centering 
\includegraphics[width=0.3\linewidth]{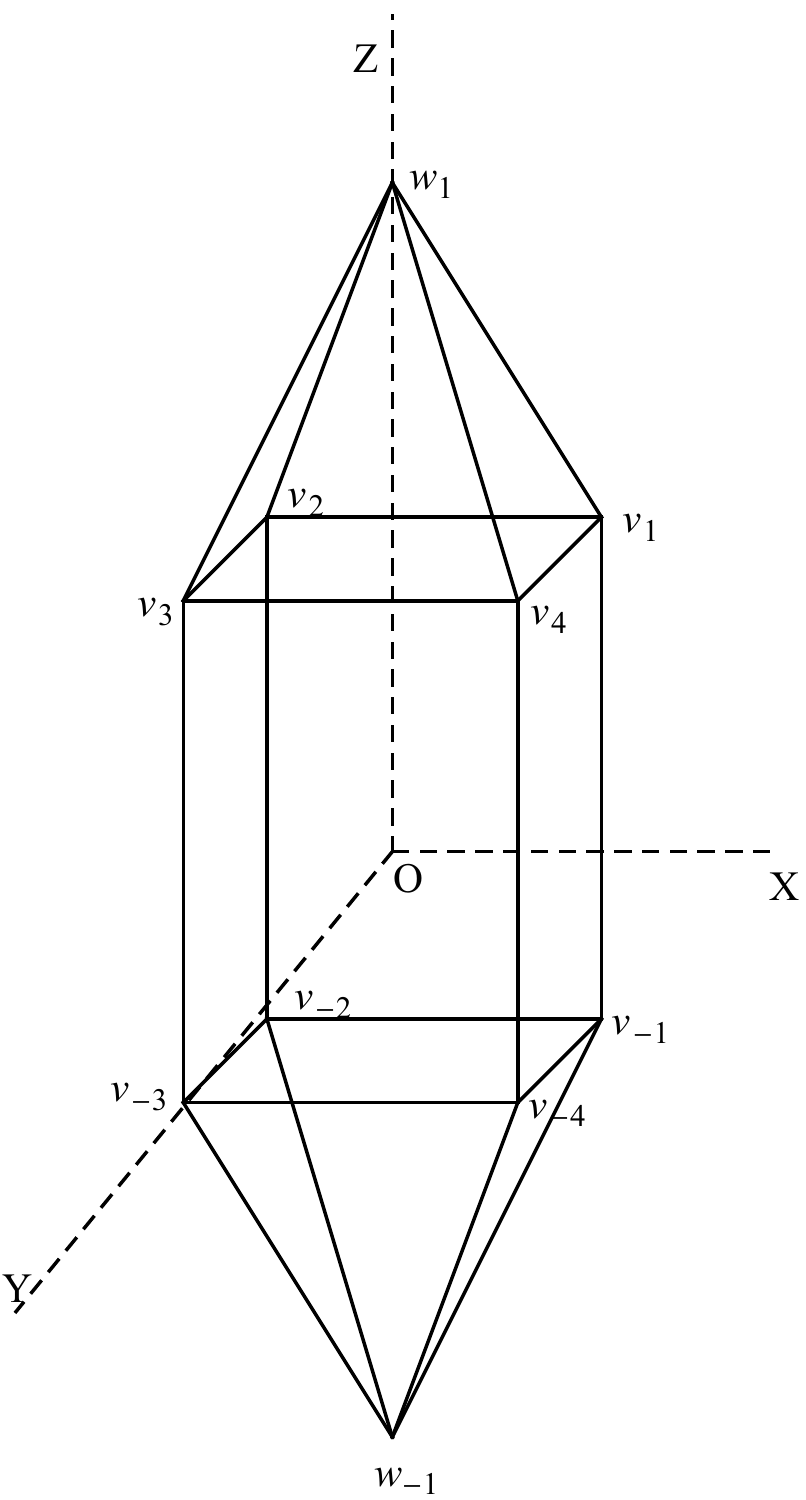}
\caption{}
\label{figure}
\end{figure}

\noindent \textbf{Case 1.} $x=w_1$ and $Tx \in \pm F_j$ for some $j \in \{1,2,3,4\}.$ Let $f_j$ be the supporting functional at $w_1$ such that $(w_1+\ker f_j) \cap S_{\mathbb{X}} = F_j.$  Then it is easy to observe that $|f_j(Tx)|=1 >\frac{1}{2}.$\\

\noindent \textbf{Case 2.} $x=w_1$ and $Tx \in  G_j$ for some $j \in \{1,2,3,4\}.$
 We assume $Tx \in G_1$, the proof for other facets $G_2, G_3, G_4$ will follow similarly. Let $Tx=(\alpha, \beta, \gamma)$ then $\beta=1, -1\leq \alpha\leq 1, -1\leq \gamma\leq 1$. Let $f_j$ be the supporting functional at $w_1$ such that $(w_1+\ker f_j) \cap S_{\mathbb{X}} = F_j.$
  Then it is easy to see that $f_1(Tx)=\frac{1}{2}+\frac{\gamma}{2}$ and $f_3(Tx)=-\frac{1}{2}+\frac{\gamma}{2}$. If $\gamma \geq 0$ then $f_1(Tx)\geq \frac{1}{2}$. If   $\gamma < 0$ then $f_3(Tx)< -\frac{1}{2}$, i.e., $|f_3(Tx)| > \frac{1}{2}$.\\

\noindent \textbf{Case 3.} $x=v_1$ and $Tx \in  G_j$ for some $j \in \{1,2,3,4\}.$ For each $j\in \{1,4\}$, let $g_j$ be the supporting functionals at $v_1$ such that $(v_1+\ker g_j) \cap S_{\mathbb{X}} = G_j.$  Since $G_1=-G_3$ and $G_2=-G_4$, so it is easy to observe that if $Tx \in G_1 \cup G_3$ then $|g_1(Tx)|=1>\frac{1}{2}$ and  if $Tx \in G_2 \cup G_4$ then $|g_4(Tx)|=1>\frac{1}{2}$.\\

\noindent \textbf{Case 4.} $x=v_1$ and $Tx \in  \pm F_j$ for some $j \in \{1,2,3,4\}.$ For each $j\in \{1,4\}$, let $g_j , f_j $ be the supporting functionals at $v_1$ such that $(v_1+\ker g_j) \cap S_{\mathbb{X}} = G_j  , (v_1+\ker f_j) \cap S_{\mathbb{X}} = F_j.$ \\

\noindent $(i)$ If $Tx\in \pm F_1$ then it is easy to observe that $|f_1(Tx)|=1>\frac{1}{2}$.\\

\noindent $(ii)$ If $Tx\in \pm F_4$ then it is easy to observe that $|f_4(Tx)|=1>\frac{1}{2}$.\\

\noindent $(iii)$ If  $Tx\in F_2$ and $Tx=(\alpha,\beta,\gamma)$ then we have, $\alpha=-\lambda-\mu, \beta=\lambda-\mu,\gamma=2-\lambda-\mu$ where $\lambda \geq 0,   \mu\geq 0,  \lambda+\mu \leq 1$. An easy calculation we have  $f_4(Tx)=1-\lambda-\mu, g_4(Tx)=-\lambda-\mu.$ Therefore if $\lambda+\mu \geq \frac{1}{2}$ then $|g_4(Tx)|\geq \frac{1}{2}$ and if $\lambda+\mu <\frac{1}{2}$ then $f_4(Tx)>\frac{1}{2}.$ Therefore if $Tx \in -F_2$ then as above $|g_4(Tx)|\geq \frac{1}{2}$ when $\lambda+\mu \geq \frac{1}{2}$ and $|f_4(Tx)|>\frac{1}{2}$ when $\lambda+\mu <\frac{1}{2}.$\\

\noindent $(iv)$ If  $Tx\in F_3$ and $Tx=(\alpha,\beta,\gamma)$ then we have, $\alpha=-\lambda+\mu, \beta=-\lambda-\mu,\gamma=2-\lambda-\mu$ where $\lambda \geq 0,   \mu\geq 0,  \lambda+\mu \leq 1$. An easy calculation we have  $f_1(Tx)=1-\lambda-\mu, g_1(Tx)=-\lambda-\mu.$ Therefore if $\lambda+\mu \geq \frac{1}{2}$ then $|g_1(Tx)|\geq \frac{1}{2}$ and if $\lambda+\mu <\frac{1}{2}$ then $f_1(Tx)>\frac{1}{2}.$ Therefore if $Tx \in -F_3$ then as above $|g_1(Tx)|\geq \frac{1}{2}$ when $\lambda+\mu \geq \frac{1}{2}$ and $|f_1(Tx)|>\frac{1}{2}$ when $\lambda+\mu <\frac{1}{2}.$\\

\noindent Thus there exists a supporting functional $x^*$ at $x$ so that $ \mid x^*(Tx) \mid \geq  \frac{1}{2}.$ Therefore, $v(T)\geq \frac{1}{2}$. This completes the proof of step 1.

\medskip 

\textbf{Step 2.} Let us define a linear operator $ T :\mathbb{X}\rightarrow \mathbb{X}$ as
\begin{eqnarray*}
T(0,0,2)&=&(1,0,0),\\
T(1,0,0)&=&(0,0,0), \\
T(0,1,0)&=&(0,0,0).
\end{eqnarray*}
It is easy to check that $Tv_i=(\frac{1}{2},0,0)\in B_{\mathbb{X}}$ for each $i\in \{1,2,3,4\}$. Since $Tw_1=(1,0,0) \in S_{\mathbb{X}}$, so $\|T\|=1$.
It follows from Theorem \ref{theorem:extreme} that  $ \textit{ext}~B_{\mathbb{X}^*} = \{g_i,f_{\pm i}~:1 \leq i \leq 4\}.$ Therefore, $v(T)=\max \{ |g_i(Tv_i)|,|g_{i-1}(Tv_i)|,|f_i(Tv_i)|,|f_{i-1}(Tv_i)|,$ $|f_j(Tw_1)|~~:~~~1\leq i,j\leq 4,g_0=g_4,f_0=f_4 \}$. Now, the four supporting functionals at $ v_j $ are $g_j, g_{j-1}, f_j, f_{j-1}$ and it is easy to see that for each $j\in \{1,2,3,4\}$, 
\[ \max \{ |g_j(Tv_j)|,|g_{j-1}(Tv_j)|,|f_j(Tv_j)|,|f_{j-1}(Tv_j)|\}=\frac{1}{2}.\]	
Similarly, the four supporting functionals at $ w_1 $ are $f_j, j\in \{1,2,3,4\}$ and it is easy to see that 
\[ \max \{|f_j(Tw_1)|~~:~~1\leq j\leq 4\}=\frac{1}{2}.\]
Therefore, $v(T)=\max \{ |g_i(Tv_i)|,|g_{i-1}(Tv_i)|,|f_i(Tv_i)|,|f_{i-1}(Tv_i)|,|f_j(Tw_1)|~~:~~~1\leq i,j\leq 4,g_0=g_4,f_0=f_4 \}=\frac{1}{2}$. This completes the proof of the theorem.
\end{proof}

 Using similar arguments, the following result can be proved:
\begin{theorem}\label{theorem:prism pyramid}
Let $\mathbb{X}$ be a $3$-dimensional polyhedral Banach space such that $B_{\mathbb{X}}$ is a polyhedron with vertices $(\cos(j-1)\frac{\pi}{n}, \sin(j-1)\frac{\pi}{n},\pm{1}), (0,0,\pm 2),  j \in \{1,2,\ldots,2n\}, n\geq 3$ . Then $n(\mathbb{X}) = \sin\frac{\pi}{2n}$,  when $n$ is odd and $n(\mathbb{X}) = \tan\frac{\pi}{2n}$,  when $n$ is even.
\end{theorem}	

We would like to end the present article with the observation that when $ n \geq 3, $ Theorem \ref{theorem:prism pyramid} gives the same result as the corresponding 2-dimensional case in Theorem $ 5 $ of \cite{MJ}. This is consistent with the fact that non-isometric Banach spaces may have the same numerical index. However, we have particularly described the case $ n=2 $ in Theorem \ref{theorem:pyramid}, as the numerical index of the concerned space turns out to be different from the case $ n=2 $ in Theorem $ 5 $ of \cite{MJ}.

\bibliographystyle{amsplain}

\end{document}